\documentclass[12pt]{birkmult}
\usepackage{amssymb}
\usepackage{amsmath}
\usepackage{version}

\newtheorem{theorem}{Theorem}[section]
\newtheorem{lemma}[theorem]{Lemma}
\newtheorem{corollary}[theorem]{Corollary}
\newtheorem{example}[theorem]{Example}
\newtheorem{definition}[theorem]{Definition}
\newtheorem{proposition}[theorem]{Proposition}
\numberwithin{equation}{section}
\def\Hol{\mathrm{Hol\, }}
\def\Hered{\mathrm{Hered\, }}
\def\Aut{\mathrm{Aut\, }}

\def\proof{\noindent {\bf Proof}. }
\def\qed{\hfill $\square$ \vspace{3mm}}
\newcommand{\T}{\mathbb{T}}
\newcommand{\D}{\mathbb{D}}
\newcommand{\Y}{\mathcal{Y}}
\newcommand{\la}{\lambda}
\newcommand{\al}{\alpha}
\newcommand{\C}{\mathbb{C}}
\newcommand{\R}{\mathbb{R}}
\newcommand{\E}{\mathcal{E}}
\newcommand{\ph}{\varphi}
\newcommand{\nn}{\nonumber}
\newcommand{\ep}{\varepsilon}
\begin{document}
\title[Magic functions and automorphisms]{The magic functions and automorphisms of a domain}
\author{J. Agler }
\address{Department of Mathematics\\
 University of California, San Diego\\
 CA92093, USA.}

\author{ N. J. Young}
\address{School of Mathematics\\
 Leeds University\\
Leeds LS2 9JT\\
 England.}
\email{N.J.Young@leeds.ac.uk}
\subjclass{   32M05, 32F45, 32F32, 46A55}
\keywords{ Complex geometry, Carath\'eodory distance, symmetrised bidisc}
\begin{abstract}
We introduce the notion of {\em magic functions} of a general domain in $\C^d$ and show that the set of magic functions of a given domain is an intrinsic complex-geometric object.  We determine the set of magic functions of the symmetrised bidisc $G$ and thereby find all automorphisms of $G$ and a formula for the Carath\'eodory distance on $G$.
\end{abstract}
\maketitle
\section{Introduction} \label{intro}
A {\em magic function} of a domain $\Omega$ in $\C^d$ is an analytic function $f$ on $\Omega$ such that the function
\[
(x,y) \mapsto 1 - \overline{f(\bar y)}f(x): \Omega\times\bar\Omega \to \C
\]
lies on an extreme ray of a certain convex cone in the space of analytic functions on $\Omega\times\bar\Omega$, where the bar denotes complex conjugation (Definition \ref{hcone}).  In this paper we show that knowledge of the magic functions of $\Omega$ has powerful consequences for the study of the geometry of $\Omega$, notably for the determination of the automorphisms of $\Omega$ and for the solution of the Carath\'eodory extremal problem.  This principle is illustrated in the case that $\Omega$ is the symmetrised bidisc $G$, defined by
$$
G = \{(z + w, zw): |z| < 1,\, |w| < 1 \} \subset \mathbb{C}^2.
$$
This domain was first studied in connection with the spectral Nevanlinna-Pick problem \cite{AY1,AY4}.  $G$ has  proved to have a rich and explicit function theory, as developed and generalised in \cite{AY6, costara, EZ} and  papers by several other authors.
In the function theory and geometry of $G$ much depends on the striking properties of certain rational functions of $3$ variables:
\begin{equation} \label{defPhi}
\Phi(z,s,p) = \frac{2zp-s}{2-zs}
\end{equation}
defined for $z,s,p \in\C$ such that $zs \neq 2$.  The  functions $\Phi(z,.)$ have sometimes been informally called the ``magic functions'' for $G$, without the term initially having  a precise meaning.  We show here that  the magic functions for $G$, in the above sense, are indeed the $\Phi(\omega,.), \, |\omega|=1$, up to composition with an automorphism of the open unit disc.  The family of magic functions for a domain $\Omega$ is invariant under automorphisms of $\Omega$; we use this fact to determine all automorphisms of $G$.   In this paper an {\em automorphism} of $\Omega$ is an analytic bijective self-map of $\Omega$ having an analytic inverse.

Our knowledge of the automorphisms of $G$ was announced in \cite[Section  6]{AY6}.  A shorter proof was found by M. Jarnicki and P. Pflug \cite{JP} and the result has been extended to the symmetrised polydisc by A. Edigarian and W. Zwonek \cite{EZ}.  The original proof, given here,  does not depend on Cartan's classification theorem for bounded homogeneous domains in $\C^2$, unlike that of \cite{JP}, but it does require some of our earlier results about $G$.  

In Section \ref{hercone} we define the hereditary cone and the magic functions of a domain; this requires a brief description of the hereditary functional calculus.  We also prove (Corollary \ref{Minvar}) the invariance of the set of magic functions under automorphisms.  In Section \ref{extreme} we describe the extreme rays of the hereditary cone of $G$ and the magic functions of $G$.  
Section \ref{autG} describes the automorphism group of $G$.  In Section \ref{cara} we explain and illustrate the use of magic functions in the Carath\'eodory extremal problem. 

Here is some notation.
 $\D,\T$ will denote the open unit disc and unit circle in $\C$ respectively.
For any domain $\Omega$  in $\mathbb{C}^d$ we denote by $\Hol \Omega$ the 
Fr\'echet algebra of holomorphic $\mathbb{C}$-valued functions on $\Omega$ 
with the topology of locally uniform convergence.  We write $\Aut\Omega$ for
the group of automorphisms of $\Omega$. By a {\em M\"obius function} we mean an element of $\Aut \mathbb{D}$. We define $\bar\Omega$ to be $\{\bar z:z\in\Omega\}$.
For Hilbert spaces $\mathcal{C}, H$ we denote by $\mathcal{L}(\mathcal{C},H), \mathcal{L}(H)$ the spaces of bounded linear operators from $\mathcal{C}$ to $H$ and from $H$ to $H$ respectively, with the operator norms.

\section{The hereditary cone and magic functions}\label{hercone}
\begin{definition}\label{hcone}
{The} {\em hereditary cone} $\Hered \Omega$ of a domain $\Omega$ 
in $\C^d$  is the set of all $h\in \Hol (\Omega \times \bar \Omega)$ such that $h(T) \ge 0$ whenever $T$ is a commuting $d$-tuple of operators on Hilbert space and $\Omega$ is a spectral domain for $T$. 
 A function $f:\Omega\to\C$ is said to be a {\em magic function} of $\Omega$ if $1-f^\vee f$ lies on an extreme ray of $\Hered\Omega$.
\end{definition}

We recall the meanings of some of the above terms.   
For any $f\in \Hol \Omega$ we define $f^\vee \in \Hol \bar \Omega$ by
$$
f^\vee(z) = \overline{f(\bar z)},\qquad z \in \bar \Omega,
$$
and $f^\vee f \in \Hol (\Omega \times \bar\Omega)$ by
$$
f^\vee f(x,y) = f^\vee(y)f(x).
$$
An {\em extreme ray}
 of a convex cone $\mathcal C$ is a set of the form $\{th: t > 0\}$ for some extreme direction $h$ of $\mathcal C$; an {\em extreme direction} of a convex cone $\mathcal C$ (in any real vector space)
is a non-zero element of $\mathcal C$ that cannot be expressed in a non-trivial way
as a sum of two elements of $\mathcal C$ -- that is, $h$ is an extreme direction of 
$\mathcal C$ if $h \in \mathcal {C}\setminus \{0\}$ and whenever $h_1, h_2 \in\mathcal C$ and $h=h_1 + h_2$ we have $h_1 = th$ 
for some $t \in \R$.

$\Omega$ is said to be a spectral domain for $T$ \cite{Ag2}
if $\sigma(T) \subset \Omega$ and, for every  bounded function $f\in \Hol \Omega$,
$$
\|f(T)\| \le \sup\limits_{z\in \Omega} |f(z)|.
$$
We must also explain the meaning of $h(T)$ where $T$ is a tuple of operators and $h\in \Hol(\Omega\times \bar\Omega)$.
An analytic function on $\Omega\times\bar\Omega$ is called a
 {\em hereditary function} on $\Omega$.  The {\em hereditary functional calculus} defines an operator $h(T)$ whenever $h$ is a hereditary function on $\Omega$ and $T$ is a commuting $d$-tuple of operators such that $\sigma(T) \subset \Omega$.  The hereditary functional calculus was introduced as a tool for the study of families of commuting tuples of operators \cite{Ag1}; here is a brief account of it.

Consider  a commuting $d$-tuple $T=(T_1, \dots, T_d)$  of $k \times k$ matrices.
Recall that $(\la_1, \dots, \la_d) \in \C^d$ is said to be a {\em joint eigenvalue}
of $T$ if there exists a non-zero vector $x \in \C^k$ such that $T_j x = \la_j x$ for 
$j=1,\dots,d.$  The {\em joint spectrum} $\sigma(T)$ of $T$ is defined to
be the set of joint eigenvalues of $T$; it is a finite non-empty subset of $\C^d$.
If all joint eigenvalues of $T$ lie in $\Omega$ then we may define, for any $h\in 
\Hol (\Omega \times \bar \Omega)$, a $k \times k$ matrix $h(T)$. (More generally,
one can define $h(T)$ when $T$ is a commuting tuple of operators with joint
spectrum contained in $\Omega$, where an {\em operator} means a
bounded linear operator on a Hilbert space; for present purposes it is enough to restrict ourselves to finite-dimensional Hilbert spaces, hence matrix tuples).
The definition in general  is somewhat technical \cite{muller}, but when $h$
is given by a locally uniformly convergent power series 
$h(x, y)= \sum c_{\alpha\beta} y^\beta x^\alpha $ on 
$\Omega \times \bar\Omega$ then $h(T)$ is defined with the usual multi-index notation by
\begin{equation}\label{defhT}
h(T) = \sum c_{\alpha\beta} (T^*)^\beta T^\alpha.
\end{equation}
Note that all unstarred matrices $T_i^{\al_i}$ are to the right of all starred
matrices ${T_j^*}^{\beta_j}$ in the definition of $h(T)$.  This definition ensures
that if $h(T) \ge 0$ and $\mathcal{M}$ is a joint invariant subspace of the
matrices $T_i$ then 
$$
h(T|\mathcal{M}) = P_\mathcal{M} h(T)|\mathcal{M} \ge 0
$$ 
where $P_\mathcal{M}$ is the operator of orthogonal projection on $\mathcal{M}$.
This ``hereditary positivity property" is the reason for the nomenclature.

In the case $\Omega =G$ we can identify a hereditary function $h$ on $G$ 
with the hereditary function $g$ on $\mathbb{D}^2$ given by
$$
g(x,y) = h(x_1+x_2, x_1 x_2, y_1+y_2,y_1y_2)
$$
where $x=(x_1,x_2)\in \D^2, y=(y_1,y_2) \in \bar\D^2.$
Since $g$ is analytic on $\D^4$ it has a locally uniformly convergent
power series expansion.  Since both $g(x,.)$ and $g(.,y)$ are symmetric
functions on $\D^2$, the partial sums of the Taylor series for $g$ can be
written in terms of the variables $(x_1+x_2, x_1 x_2)$ and $(y_1+y_2, y_1 y_2)$,
which amounts the statement that $h$ can be locally uniformly approximated
by polynomials on $G \times \bar G$, so that the formula (\ref{defhT}) applies.

Observe that $\Hol (\Omega \times \bar \Omega)$ is a module 
over both $\Hol \Omega$ and $\Hol \bar \Omega$ in a natural way.  
We shall write the $\Hol \bar \Omega$ action on the left and the 
$\Hol \Omega $ action on the right:  thus, if $h\in \Hol (\Omega \times \bar \Omega)$, $f\in \Hol \Omega$ and $g\in \Hol \bar \Omega$ we define 
$g\cdot h \cdot f \in \Hol (\Omega \times \bar \Omega)$ by
$$
g\cdot h\cdot f(x, y) = g(y) h(x, y) f(x), \qquad x \in \Omega, \ y\in \bar \Omega.
$$
 
A property of the hereditary functional calculus is: if $h\in \Hol (\Omega \times \bar \Omega), f \in \Hol \Omega, \, g \in \Hol \bar\Omega$ and $\sigma(T) \subset \Omega$ then
\begin{equation}\label{conjug}
(g \cdot h \cdot f)(T) = g^\vee(T)^* h(T) f(T).
\end{equation}
These rudiments of the hereditary functional calculus are well established \cite{Ag1}. 

We shall state some simple consequences of the definition of the hereditary cone. Recall that, for any set $S$, a function $k: S\times S \to \mathbb{C}$ is said to be  {\em positive semi-definite} if $\sum\limits^n_{i, j=1} k(s_i, s_j) c_i \bar c_j \ge 0$ whenever $n$ is a positive integer, $s_1, \dots, s_n \in S$ and $c_1, \dots, c_n \in \mathbb{C}$.
For a domain $\Omega$  in $\mathbb{C}^d$, $\mathcal{P}(\Omega)$ will denote the cone of functions $h\in\Hol(\Omega \times \bar \Omega)$ such that the map
 $$
(\lambda, \mu) \mapsto h(\lambda, \bar \mu)
$$
 is a positive semi-definite function on $\Omega$.
 We shall need the following  theorem of Aronszajn \cite{Aron} on positive definite functions.
 \begin{proposition}\label{exposdef}
A function $h:\Omega\times\Omega \to\C$ is positive semi-definite if and only if there exists a Hilbert space $\mathcal E$ and a function $F:\Omega\to \mathcal E$ such that
\[
 h(\la,\mu) = \langle F(\la), F(\mu) \rangle \mbox{  for all } \la,\mu \in \Omega.
\]
A non-zero function $h\in \mathcal{P}(\Omega)$ lies on an extreme ray of $\mathcal{P}(\Omega)$ if and only if there exists an analytic function $F: \Omega \to \mathbb{C}$ such that 
 $$
h(\lambda, \mu) = F^\vee(\mu) F(\lambda)\quad \mbox{for all}\quad \lambda \in \Omega, \mu \in \bar \Omega.
 $$
 \end{proposition}
\begin{proposition}\label{simpleprops}
\begin{enumerate}
\item [\rm(1)]$\Hered \Omega$ is a closed convex cone in $\Hol (\Omega \times \bar\Omega)$;  
\item[\rm(2)] $\Hered \Omega$ is closed under conjugation by any element $g$ of $\Hol \Omega$: if $h\in \Hered \Omega$ then $g^\vee\cdot h\cdot g \in \Hered \Omega$;
\item [\rm(3)] $\mathcal{P}(\Omega)\subset  \Hered\Omega $;
\item[\rm(4)] for $f\in\Hol\Omega, \, 1-f^\vee f \in\Hered\Omega$ if and only if $|f| \leq 1$ on $\Omega$.
\end{enumerate}
\end{proposition}
\begin{proof}
(1), (2) and (4) are easy: note that by equation (\ref{conjug}),
\[
(g^\vee\cdot h\cdot g)(T)=g(T)^*h(T)g(T).
\]
(3) Consider $h\in\mathcal{P}(\Omega)$.  By Proposition \ref{exposdef} there exist a Hilbert space $\mathcal{E}$ and an analytic function $F:\Omega\to\E$ such that $h(\la,\bar\mu)=\left<F(\la),F(\mu)\right>$ for all $\la, \mu\in\Omega$.  By the standard functional calculus, for any Hilbert space $H$ and any commuting $d$-tuple $T$ of operators on $H$ such that $\sigma(T)\subset\Omega$, the operator $F(T)$ is defined and acts from $H$ to $\E\otimes H$.  Moreover $h(T)=F(T)^*F(T)\geq 0$.  In particular this is true when $\Omega$ is a spectral domain for $T$, and so $h\in\Hered\Omega$. \qed
\end{proof}

We call property (2) {\em conjugacy-invariance}.
We shall say that a subset $E$ of a closed conjugacy-invariant cone $\mathcal{C}$ {\em generates} $\mathcal{C}$ if $\mathcal{C}$ is the smallest closed conjugacy-invariant convex cone that contains $E$.
For example,  the set comprising the constant function $1$ generates the cone $\mathcal{P}(\Omega)$.

In the following example we summarise some well-known facts about the disc.
\begin{example}
\label{heredD}
The hereditary cone of the unit disc $\D$ is generated by the single hereditary polynomial
$h_0(x,y)=1-yx$.  A non-zero hereditary function $h$ on $\D$ lies on an extreme ray of $\Hered\D$ if and only if $h=g^\vee \cdot h_0 \cdot g$ for some non-zero $g\in\Hol\D$.
\end{example}
By von Neumann's inequality, $\D$ is
a spectral domain for an operator $T$ if and only if $\sigma(T) \subset \D$ and
$||T|| \le 1$.  For any such $T$,  $h_0(T) =1-T^*T\geq 0$ and so  $h_0\in\Hered \D$.
\begin{lemma}
The operator
\[
M_{h_0}: \Hol(\D\times\bar\D) \to \Hol(\D\times\bar\D): h\mapsto h_0 h
\]
is a $\Hol\bar\D - \Hol\D$ module morphism that maps $\mathcal{P}(\D)$ bijectively onto $\Hered\D$.
\end{lemma}
\begin{proof}
Clearly $M_{h_0}$ is a morphism of bimodules and is bijective.
  Consider $h \in\Hered \D$.  Let $k_y$ be the Szeg\H o kernel on $\D$:
$$
k_y(x) = \frac{1}{1 - \bar y x}, \qquad x,y \in \D,
$$
so that $k_y$ is the reproducing kernel for $y$ in the Hardy space $H^2$.
Let $S^*$ denote the backward shift operator on $H^2$.  By virtue of the
fact that $k_y$ is an eigenvector of $S^*$ with eigenvalue $\bar y$ we
have, for any $r \in (0,1)$ and $x,y \in \D$,
$$
\left<h(rS^*) k_x, k_y\right> = h(r\bar x, ry)\left<k_x, k_y\right>.
$$
$\D$ is a spectral domain for $rS^*$, and so for the 
vector $ f= \sum c_i k_{x_i} \in H^2$, we have
$$
0 \le \left<h(rS^*)f,f\right> = \sum_{i,j} c_i \bar c_j h(r\bar x_i, rx_j) \left<k_{x_i}, k_{x_j}\right>,
$$
which is to say that the function 
$$
f(x,y) = \frac{h(rx,ry)}{1-yx}
$$
is positive semi-definite on $\D$ for $r < 1$, and hence also for $r=1$.   Thus $f\in
\mathcal{P}(\D)$ and $M_{h_0}f=h$, so that $M_{h_0}\mathcal{P}(\D)\supset\Hered\D$.

Suppose $f\in\mathcal{P}(\D)$.
By Proposition \ref{exposdef} there exists an analytic function $g :\D\to \ell^2$ such that
$$
f(x,y) = \left<g(x), g^\vee(y)\right>
$$
for all $x \in \D, y \in \bar \D$.  It follows that $h_0 f$ is the limit in $\Hol(\D\times\bar\D)$ of functions of the form $\sum g_j^\vee \cdot h_0 \cdot g_j$, and hence that $h_0 f \in\Hered\D$.  Thus   $M_{h_0}\mathcal{P}(\D)\subset\Hered\D$. \qed
\end{proof}

It follows from the lemma and Proposition \ref{exposdef} that $\Hered\D$ is generated by the set $\{M_{h_0} 1\}=\{h_0\}$ and that the points on the extreme rays of $\Hered\D$ are the functions $M_{h_0} g^\vee g= g^\vee\cdot h_0\cdot g, \, g\in \Hol\D$.

We observe also that $\Hered\D$ is not closed under pointwise multiplication.   Let $T$ be an operator such that $T^2=0$ and $1/\sqrt{2} < ||T|| \leq 1$:  then $\D$ is a spectral domain for $T$ but $h_0^2(T)=1-2T^*T \not\geq 0$.  Hence $h_0^2 \notin \Hered\D$.
\qed

Similar results hold for the bidisc: $\Hered\D^2$ is generated by the pair of functions $\{1-y_1x_1, 1-y_2 x_2\}$.   Things are not so simple for the tridisc, because of the failure of von Neumann's inequality for $\D^3$ \cite{Ag3}.

\begin{example} \label{heredCd}
$\Hered\C^d = \mathcal{P}(\C^d)$.  The functions lying on the extreme rays of $\Hered\C^d$ are those of the form $g^\vee g$ for some $g\in\Hol\C^d$.
\end{example}
Indeed, $\C^d$ is a spectral domain for {\em every} commuting $d$-tuple of operators.
Consider $h\in\Hered\C^d$ and choose $c_1,\dots,c_N\in\C$ and points $\la_j=(\la_j^1, \dots,\la_j^d)\in\C^d, \, 1\leq j\leq N$.   Let $T_i=\mathrm{diag}(\la_1^i, \dots,\la_N^i) \in \C^{N\times N}, \, 1\leq i\leq d,$ and let $T=(T_1,\dots,T_d), \, c=[c_1 \dots c_N]^T \in \C^N$.  Then
\[
0 \leq \left< h(T)c,c\right> = \sum_{i,j=1}^N c_i\bar c_j h(\la_i, \bar\la_j).
\]
Thus $h\in \mathcal{P}(\C^d)$.  Hence $\Hered\C^d = \mathcal{P}(\C^d)$.  The second statement follows from Proposition \ref{exposdef}. \qed

From Example \ref{heredD} one can easily show that the magic functions of $\D$ are precisely the automorphisms of $\D$.
From Example \ref{heredCd}  the magic functions of $\C^d$ are the constant functions of modulus less than one.  The magic functions of the bidisc are the functions $(z_1,z_2)\mapsto  m( z_j)$  where $m\in \Aut\D$ and $j=1,2$.

The next statement shows that $\Hered$ is a contravariant functor from the category of domains and analytic maps to the category of convex cones and linear maps.
 \begin{proposition}\label{functor}
  Let $\Omega_1\subset\C^q,\ \Omega_2\subset \C^d$ be domains, let $\alpha:\Omega_1\to \Omega_2$ be analytic and let $\alpha^\vee : \bar \Omega_1 \to \bar\Omega_2$ be given by 
  $$
\alpha^\vee (z) = \overline{\alpha(\bar{z})},\qquad z \in \bar \Omega_1.
$$
  The map
  $$
\alpha^\# : h \mapsto h \circ (\alpha \times \alpha^\vee)
$$
  is a linear map from $\Hol (\Omega_2 \times \bar \Omega_2)$  to $\Hol (\Omega_1 \times \bar \Omega_1)$
that maps the cone $\Hered \Omega_2$ into $\Hered\Omega_1$.  
If $\al(\Omega_1)$ is open in $\Omega_2$ and $\al^\#h$ is an extreme direction of $\Hered\Omega_1$ then $h$ is an extreme direction of $\Hered\Omega_2$.
  \end{proposition}
  
  \proof
  It is clear that $\alpha \times \alpha^\vee$ is analytic on $\Omega_1 \times \bar \Omega_1$, and hence that $\alpha^\#$ is a linear map.  Consider $h\in \Hered \Omega_2$ and any 
commuting $q$-tuple $T$ of operators such that $\sigma(T) \subset \Omega_1$ and 
$\Omega_1$ is a spectral domain for $T$.  If $\alpha = (\alpha_1, \dots, \alpha_d)$ then the Taylor functional calculus enables us to define $\alpha(T) = (\alpha_1(T), \dots, \alpha_d(T))$ as a commuting $d$-tuple of operators.   Moreover, by the Spectral Mapping Theorem (for example \cite[Theorem 2.5.10]{EP}),
  $$
\sigma(\alpha(T)) = \alpha(\sigma(T)) \subset \Omega_2.
$$
  It is immediate that $\Omega_2$ is a spectral domain for $\alpha(T)$, and hence
  $$
(\alpha^\#h) (T) = h \circ (\alpha \times \alpha^\vee) (T) = h(\alpha (T)) \ge 0.
$$
  Hence $\alpha^\# h \in \Hered \Omega_1$.  

Suppose that $\al^\#h$ is an extreme direction of $\Hered \Omega_1$ and $h=g+k$ where $g,k\in \Hered \Omega_2$.  Then $\al^\#h=\al^\#g+\al^\#k$ and $\al^\#g,\al^\#k \in \Hered\Omega_1$.  Hence $\al^\#h=t\al^\#g$ for some $t>0$, which is to say that $h=tg$ on the open set $\al(\Omega_1)\times\al^\vee(\bar\Omega_1)$.  By the connectedness of $\Omega_2$ it follows that $h=tg$.  Hence $h$ is an extreme direction of $\Hered\Omega_2$.
 \qed
\begin{corollary}\label{fcircal}
Let $\al:\Omega_1\to\Omega_2$ be an analytic map with open range.  If $f\in\Hol \Omega_2$ and $f\circ\al$ is a magic function of $\Omega_1$ then $f$ is a magic function of $\Omega_2$.
\end{corollary}
\begin{corollary}\label{permutes}
If $\al:\Omega_1 \to \Omega_2$ is an isomorphism of domains then $\al^\#:\Hol(\Omega_2\times \bar\Omega_2) \to \Hol( \Omega_1\times\bar\Omega_1)$ is a linear isomorphism that maps $\Hered \Omega_2$ onto $\Hered \Omega_1$ and maps extreme rays of $\Hered \Omega_2$ to extreme rays of $\Hered \Omega_1$.
\end{corollary}
 The following statement, to the effect that the notion of magic function is an intrinsic complex-geometric one, follows from either of Corollaries \ref{fcircal} or \ref{permutes}.
\begin{corollary}\label {Minvar}
Isomorphisms preserve magic: if $\alpha:\Omega_1\to \Omega_2$ is an isomorphism of domains and $f$ is a magic function of $\Omega_2$ then $f\circ \al$ is a magic function of $\Omega_1$.  
\end{corollary}

Here is a straightforward invariance property of magic functions.
\begin{proposition}
If $f$ is a magic function on a domain $\Omega$ then so is $m\circ f$ for any $m\in \Aut\D$.
\end{proposition}
\proof
Let $m(\la)=\omega (\la-\al)/(1-\bar\al \la),\, \omega\in\T, \, \al\in\D$.  Then we have
\[
1-(m\circ f)^\vee(m\circ f) = g^\vee \cdot (1-f^\vee f)\cdot g
\]
where $g$ is an invertible function in $\Hol \Omega$ given by 
\[
g(x)=(1-|\al|^2)^{\tfrac 12} (1-\bar\al f(x))^{-1}.
\]
Since $1-f^\vee f$ is an extreme direction in $\Hered \Omega$, so is $g^\vee \cdot (1-f^\vee f)\cdot g$.  Thus $m\circ f$ is magic on $\Omega$.
\qed

 \section{Extreme rays of the hereditary cone of $G$} \label{extreme}
\begin{theorem}\label{exrays}
 A hereditary function $h\in \Hered G$ lies on an extreme ray of $\Hered G$ if and only if $h$ is expressible in the form
 \begin{equation}\label{ray}
 h=f^\vee\cdot\left(1- \Phi^\vee_\omega \Phi_\omega\right) \cdot f
 \end{equation}
 for some $f\in \Hol  G$ and $\omega \in \mathbb{T}$, where $\Phi_\omega : G \to \mathbb{C}$ is defined by
 $$
\Phi_\omega (s, p) = \Phi(\omega,s,p)=\frac{2\omega p -s}{2-\omega s},\qquad (s, p) \in G.
$$
 \end{theorem}
Recall that $\Phi_\omega$ maps $G$ to $\D$, e.g. \cite[Theorem 2.1]{AY6}, so that $1- \Phi^\vee_\omega \Phi_\omega \in\Hered G$ for $\omega\in\T$.
  The proof will be based on the following result, which is a straightforward consequence of \cite[Theorem 3.5]{AY1}.

\begin{theorem}\label{realn}
 A hereditary function $h$ on $G$ belongs to $\Hered G$ if and only if there exist a separable Hilbert space $H$, an $\mathcal{L}(H)$-valued spectral measure $E$ on $\mathbb{T}$  and a continuous map $u:\mathbb{T} \times G \to H$ such that 
 \begin{itemize}
 \item[\rm{(i)}]   $\ u(\omega, .): G\to H$ is analytic for every $\omega \in \mathbb{T}$ and $u(.,\la)$ satisfies a Lipschitz condition on $\T$, uniformly for $\la$ in any compact subset of $G$;

 \item[\rm{(ii)}] for  all $\lambda, \mu \in G$,
 $$
h(\lambda, \bar \mu) = \int_{\mathbb{T}} \big(1 -  \Phi^\vee_\omega (\bar \mu) \Phi_\omega(\lambda)\big)
 \langle E(d\omega)u(\omega, \lambda), u(\omega, \mu) \rangle
 $$
where the integral is a Riemann-Stieltjes integral that converges uniformly for $(\la,\bar\mu)$ in any compact subset of $G\times\bar G$.
 \end{itemize}
 \end{theorem}
\noindent {\bf Note.}  This statement differs slightly from that of \cite[Theorem 3.5]{AY1}.
Firstly, we have specialised from operator-valued to scalar-valued $h$.  Secondly, that theorem was stated for the function
\[
 (x, \bar y) \mapsto h(x_1+x_2,x_1x_2,\bar y_1+\bar y_2,\bar y_1\bar y_2): \D^2\times\D^2 \to \C
\]
rather than for $h$ on $G\times \bar G$ as here.  Thirdly, that theorem used the notation $\nu_\omega(x,y)$ where
\begin{eqnarray*}
 \nu_{\bar\omega}(x,y) &=& 2(1-y_2x_2) + \bar\omega(y_2x_1 - y_1) + \omega (y_1x_2-x_1) \\
  &=&  \tfrac 12 (2-\omega y_1)(2- \bar\omega x_1) (1 - \Phi_\omega^\vee(y)\Phi_\omega(x)).
\end{eqnarray*}
Here we have absorbed the factor $\mathrm{const}\cdot(2- \bar\omega x_1)$ into $u(\omega,x)$.  Fourthly, the facts that the Hilbert space $H$ can be taken to be separable and that $u(.,\la)$ is Lipschitz were not explicitly stated, but they follow easily from the proof of \cite[Theorem 3.5]{AY1}.
\begin{corollary} \label{compactgen}
The set
$$
\mathcal{Y}=\{1- \Phi^\vee_\omega \Phi_\omega : \omega \in \T \}
$$
is a compact generating set for $\Hered G$.
\end{corollary}
\begin{proof}
It is clear that $\Y \subset\Hered G$ and that $\mathcal{Y}$ is a continuous image of $\T$, hence is compact.  For $E$ and $u$ as in Theorem \ref{realn} the map
\[
(\la,\bar\mu) \mapsto \langle E(\tau) u(\omega,\la), u(\omega,\mu) \rangle: G\times \bar G \to \C
\]
is positive semi-definite on $G$ for any interval $\tau$ in $\T$ and any $\omega\in\T$.  It  can therefore be written $\langle f(\la), f(\mu) \rangle$ for some analytic $f:G\to \ell^2$,
and hence the function
\begin{equation} \label{summand}
  (\la,\bar\mu)\mapsto (1-\Phi^\vee_\omega(\bar\mu)\Phi_\omega(\la)) \langle E(\tau) u(\omega,\la), u(\omega,\mu) \rangle
\end{equation}
belongs to the closed conjugacy-invariant cone generated by $\mathcal{Y}$ for any $\tau,\omega$.  

Consider any $h\in \Hered G$.  By the definition of the Riemann integral in Theorem \ref{realn},  $h$ can be approximated uniformly on compact subsets of $G\times \bar G$ by finite sums of functions of the form (\ref{summand}).  Hence $h$ is in the cone generated by $\mathcal{Y}$, and so
$\Y$ is a generating set for $\Hered G$. \qed
\end{proof}

We shall show below in Proposition \ref{minimal} that $\mathcal{Y}$ is in fact a {\em minimal} closed generating set for $\Hered\Omega$.

We require two properties of the slightly unusual integral on the right hand side of (ii) in Theorem \ref{realn}, to wit existence and a form of Cauchy-Schwarz inequality.  These can hardly be new, but we do not know a reference for them.

\begin{proposition} \label{integral}
Let $\mathcal{C}, H$ be separable Hilbert spaces, let $E$ be an $\mathcal{L}(H)$-valued spectral measure on $\T$ and let $f,g :\T\to\mathcal{L}(\mathcal{C},H)$ be functions that satisfy a uniform Lipschitz condition on $\T$.  Then the Riemann-Stieltjes integral
\begin{equation}\label{gEf}
 \int_\T g(\omega)^*E(d\omega) f(\omega)
\end{equation}
converges in norm to an element of $\mathcal{L}(\mathcal{C})$.
\end{proposition}
\begin{proof}
Suppose that $||f(\omega)|| \leq M, \, ||g(\omega)||\leq M$ for $\omega\in\T$ and that$f,g$ satisfy a Lipschitz condition with constant $K>0$, in the sense that $||f(\omega_1)- f(\omega_2)|| \leq K d( \omega_1,\omega_2)$ where $d$ is the normalised arc length metric on $\T$. Let $\ep > 0$; we shall show that there is a partition $\tau=(\tau_1, \dots,\tau_n)$ of $\T$ such that the Riemann-Stieltjes sums approximating the integral (\ref{gEf}) corresponding to any pair of refinements of $\tau$ differ by at most $\ep$.    

Choose $\delta$ so that
\[
0 < \delta <  \left(\frac{\ep}{8K(M+K)}\right)^2.
\]
Pick a partition $\tau: \tau_1\cup\dots\cup\tau_n$ of $\T$ such that the normalised arc length $\delta_j$ of $\tau_j$ is at most $\delta$.  Corresponding to $\tau$ and a choice $\xi=(\xi_1,\dots,\xi_n)$ with $\xi_j\in\tau_j$ the approximating Riemann-Stieltjes sum to the integral (\ref{gEf}) is defined to be
\[
 S(\tau,\xi) = \sum_{j=1}^n g(\xi_j)^*E(\tau_j)f(\xi_j).
\]
We claim that, for any refinement $\sigma:\sigma_1\cup\dots\cup\sigma_m$ of $\tau$ and any choice of $\xi=(\xi_1,\dots,\xi_n), \eta=(\eta_1,\dots,\eta_m)$ with $\xi_j\in\tau_j, \eta_i\in\sigma_i$,
\[
 ||S(\tau,\xi) - S(\sigma,\eta)||_{\mathcal{L}(\mathcal{C})} < \frac{\ep}{2}.
\]
From this it will follow that $\tau$ has the claimed property,  hence that the net $S(\tau,\xi)$ is Cauchy with respect to the operator norm on the complete space $\mathcal{L}(\mathcal{C})$ and hence that the integral (\ref{gEf}) converges.

For $1 \leq i\leq m$ let $i'$ denote the index $j\in\{1,\dots, n\}$ such that $\sigma_i \subset \tau_j$.  Then, since $E(\tau_j) = \sum_{\sigma_i \subset \tau_j} E(\sigma_i)$,
\begin{eqnarray}\label{diffS}
S(\tau,\xi) - S(\sigma,\eta) &=& \sum_{j=1}^n g(\xi_j)^*E(\tau_j)f(\xi_j) -\sum_{i=1}^m g(\eta_i)^*E(\sigma_i)f(\eta_i)\nn \\
   &=&  \sum_{i=1}^m  \{g(\xi_{i'})^* E(\sigma_i) f(\xi_{i'})- g(\eta_i)^*E(\sigma_i)f(\eta_i)\} \nn\\
  &=& \sum_{i=1}^m \left\{ g(\xi_{i'})^*E(\sigma_i)[f(\xi_{i'}) - f(\eta_i)] + [ g(\xi_{i'})-g(\eta_i)]^*E(\sigma_i)f(\eta_i)\right\}\nn\\
  &=& \left(\sum_{i=1}^m  g(\xi_{i'})^*E(\sigma_i)\right) \left(\sum_{i=1}^m E(\sigma_i)[f(\xi_{i'}) - f(\eta_i)]\right) + \nn\\
 & & \qquad \left(\sum_{i=1}^m  [g(\xi_{i'}-g(\eta_i)]^*E(\sigma_i)\right)\left(\sum_{i=1}^m E(\sigma_i)f(\eta_i)\right),
\end{eqnarray}
the last step because $E(\sigma_i)E(\sigma_j) =0$ when $i\neq j$.  We shall estimate the norms of the four operators in round brackets with the aid of the partial summation formula.
 If $X_1,\dots,X_m \in \mathcal{L}(\mathcal{C},H)$ then
\begin{eqnarray*}
|| \sum_{i=1}^m E(\sigma_i)X_i || &=& || X_1 + \sum_{j=2}^m E(\sigma_j\cup\dots\cup\sigma_m)(X_j-X_{j-1})|| \\
 &\leq& ||X_1|| +  \sum_{j=2}^m ||X_j-X_{j-1}||.
\end{eqnarray*}
Hence, if we arrange the $\xi_j, \eta_i$ in order of increasing arguments, we have
\begin{eqnarray}\label{in0}
||\sum_{i=1}^m g(\xi_{i'})^*E(\sigma_i)|| &=& ||\sum_{j=1}^n g(\xi_j)^*E(\tau_j)|| \nn\\
 &\leq&  ||g(\xi_{1})|| +  \sum_{j=2}^n || g(\xi_{j}) - g(\xi_{j-1})|| \nn\\
 &\leq& M+K \sum_{j=2}^n d(\xi_j,\xi_{j-1}) \nn\\
 &\leq& M+K.
\end{eqnarray}
Moreover, if $\sigma_i\subset\tau_j$ precisely when $\ell(j)\leq i\leq u(j)$, then
\begin{eqnarray}\label{in1}
 \sum_{i=1}^m E(\sigma_i)[f(\xi_{i'}) - f(\eta_i)] &=& \sum_{j=1}^n \sum_{i=\ell(j)}^{u(j)}E(\sigma_i)[f(\xi_{j}) - f(\eta_i)] \nn\\
  &=& \sum_{j=1}^n  E(\tau_j)[f(\xi_j)-f(\eta_{\ell(j)})] + \nn\\
  \sum_{j=1}^n \sum_{i=\ell(j)+1}^{u(j)} & &E(\sigma_i\cup\dots\cup\sigma_{u(j)})[f(\eta_{i-1}) - f(\eta_i)] \}
\end{eqnarray}
Since the operators $E(\tau_j)[f(\xi_j)-f(\eta_{\ell(j)})]$ have pairwise orthogonal ranges, we have
\begin{eqnarray} \label{in2}
 || \sum_j E(\tau_j)[f(\xi_j)-f(\eta_{\ell(j)})]||^2  & \leq &  \sum_j ||E(\tau_j)[f(\xi_j)-f(\eta_{\ell(j)})]||^2 \nn\\
  & \leq & \sum_j K^2 d(\xi_j,\eta_{\ell(j)})^2 \leq K^2 \sum_j \delta_j^2\nn\\
  &\leq & K^2\delta\sum_j \delta_j \leq K^2\delta.
\end{eqnarray}
Now
\begin{eqnarray*}
 ||\sum_{i=\ell(j)+1}^{u(j)}E(\sigma_i\cup\dots\cup\sigma_{u(j)})[f(\eta_{i-1}) - f(\eta_i)]|| &\leq & \sum_{i=\ell(j)+1}^{u(j)}||f(\eta_{i-1}) - f(\eta_i)|| \\
  &\leq & K \sum_{i=\ell(j)+1}^{u(j)} d(\eta_{i-1}, \eta_i) \leq K\delta_j,
\end{eqnarray*}
and again by orthogonality of ranges,
\begin{eqnarray} \label{in3}
 & & ||\sum_{j=1}^n \sum_{i=\ell(j)+1}^{u(j)}E(\sigma_i\cup\dots\cup\sigma_{u(j)})[f(\eta_{i-1}) - f(\eta_i)]||^2 \nn\\
&\leq &\sum_{j=1}^n ||\sum_{i=\ell(j)+1}^{u(j)}E(\sigma_i\cup\dots\cup\sigma_{u(j)})[f(\eta_{i-1}) - f(\eta_i)]||^2 \nn\\
 &\leq & \sum_{j=1}^n \left(\sum_{i=\ell(j)+1}^{u(j)} ||f(\eta_{i-1}) - f(\eta_i)||\right)^2\nn \\
  &\leq& \sum_{j=1}^n \left(\sum_{i=\ell(j)+1}^{u(j)} Kd( \eta_{i-1},\eta_i)  \right)^2\nn \\
  &\leq & \sum_{j=1}^n K^2 \delta_j^2 \leq K^2\delta\sum_{j=1}^n \delta_j \nn\\
 &\leq& K^2\delta.
\end{eqnarray}
On combining equation (\ref{in1}) with inequalities (\ref{in2}) and (\ref{in3})  we find
\[
 ||\sum_{i=1}^m E(\sigma_i)[f(\xi_{i'} - f(\eta_i)]||  \leq 2K\sqrt{\delta}.
\]
Putting this relation together with inequality (\ref{in0}) we have
\[
 ||\sum_{i=1}^m  g(\xi_{i'})^*E(\sigma_i)\cdot \sum_{i=1}^m E(\sigma_i)[f(\xi_{i'}) - f(\eta_i)]|| \leq (M+K)\cdot 2K\sqrt{\delta}.
\]
The same estimate applies to the second term on the right hand side of equation (\ref{diffS}),
and hence
\begin{eqnarray*}
||S(\tau,\xi) - S(\sigma,\eta)|| & \leq & 4K(M+K)\sqrt{\delta} \\
 & < \frac{\ep}{2}
\end{eqnarray*}
as claimed.  Hence the integral (\ref{gEf}) converges in norm.  \qed
\end{proof}
\noindent {\bf Remark} (i)  The proof shows that if $\mathcal{F}$ is a family of functions from $\T$ to  $\mathcal{L}(\mathcal{C},H)$ that is uniformly bounded in Lipschitz norm then the integral (\ref{gEf}) converges in norm {\em uniformly} for $f,g\in\mathcal{F}$.\\
(ii) Mere continuity of $f$ and $g$ does not suffice for the convergence of the integral (\ref{gEf}).  Let $\mathcal{C}=\C$, let $H=L^2(0,2\pi)$ and let $E(\delta)$ be the operation of multiplication by the characteristic function $\chi_\delta$ of $\delta$ for any measurable $\delta\subset\T$.  Let $f:\T\to H$ be defined by
\[
 f(\mathrm{e}^{i\theta}) = \chi_{(0,\theta)}, \quad 0 < \theta < 2\pi.
\]
Then $f$ is continuous but the integral
\[
 \int_\T \langle E(d\omega)f(\omega), f(\omega) \rangle
\]
diverges.  Indeed, if $\tau$ is any partition of $\T$ and $0 < t < 2\pi$ then there exists a choice of $\kappa_j\in\tau_j$ such that $S(\tau,\kappa) = t$.  For example, if $\kappa_j$ is taken to be the mid-point of $\tau_j$ then $S(\tau,\kappa) = \pi.$
\begin{proposition}\label{causch}
If  $H$  is a separable Hilbert space, $E$ is an $\mathcal{L}(H)$-valued spectral measure on $\T$  and $f,\, g:\T \to H$ satisfy uniform Lipschitz conditions then, for any interval $J\subset \T$ ,
\[
\left| \int_J  \langle E(d\omega) f(\omega),g(\omega) \rangle  \right| \leq  
  \left\{\int_J \langle E(d\omega) f(\omega), f(\omega) \rangle \right\}^{\tfrac 12}
 \left\{\int_J \langle E(d\omega) g(\omega), g(\omega) \rangle \right\}^{\tfrac 12}.
 \]
\end{proposition}
\begin{proof}
Corresponding to a partition $\tau=\{\tau_1, \dots, \tau_N\}$ of $J$ and a choice of points
$\kappa_j\in\tau_j$, define the approximating Riemann-Stieltjes sum
\[
S(\tau,\kappa) = \sum_{j=1}^N \langle E(\tau_j)f(\kappa_j), g(\kappa_j) \rangle
\]
to the integral on the left hand side.  We have
\begin{eqnarray*}
|S(\tau,\kappa)| &\leq& \sum_j | \langle E(\tau_j)f(\kappa_j), g(\kappa_j) \rangle | \\
	&\leq &  \sum_j ||E(\tau_j)f(\kappa_j)|| \cdot ||E(\tau_j)g(\kappa_j)||\\
        &\leq & \left\{\sum_j ||E(\tau_j)f(\kappa_j)||^2\right\}^{\tfrac 12} \left\{\sum_j ||E(\tau_j)g(\kappa_j)||^2\right\}^{\tfrac 12} \\
	&=&  \left\{\sum_j \langle E(\tau_j)f(\kappa_j), f(\kappa_j) \rangle\right\}^{\tfrac 12} \left\{\sum_j  \langle E(\tau_j)g(\kappa_j), g(\kappa_j) \right\}^{\tfrac 12}.
\end{eqnarray*}
On taking limits with respect to refinement of the partition $\tau$ we obtain the inequality in the lemma whenever the integrals in question exist. In particular, by Proposition \ref{integral} specialised to $\mathcal{C}=\C$, this is so when $f,g$ satisfy uniform Lipschitz conditions.
\qed
\end{proof}

 The variety $\{ (2\la,\la^2): \la \in \C \}$ will play a special role: we call it the
{\em royal variety} and denote it by $\mathcal V$.

\noindent {\bf Proof of Theorem \ref{exrays}.}  Consider a function $h$ on $G \times \bar G$ of the form
$$
h = f^\vee \cdot (1 - \Phi^\vee_\omega \Phi_\omega) \cdot f
$$
where $f \in \Hol G$ and $\omega \in \T$.   This function $h$ is analytic on $G \times \bar G$, and indeed $h \in \Hered G$:  if $T$ is a commuting pair
of operators such that $\sigma(T) \subset G$  and $G$ is a spectral domain for $T$,
then since $\Phi_\omega$ is bounded by $1$ on $G$ we have $||\Phi_\omega(T)|| \le 1$, and hence
$$
h(T) = f(T)^*(1- \Phi_\omega(T)^*\Phi_\omega(T))f(T) \ge 0.
$$
We shall show that $h$ lies on an extreme ray of $\Hered G$.
Suppose that $h=h_1+h_2$ where $h_1,h_2 \in  \Hered G$, so that,
for all $\la, \mu \in G$,
\begin{equation}
\label{relh}
\bar f(\mu) (1- \overline{\Phi_\omega(\mu)}\Phi_\omega(\la))f(\la) = h_1(\la, \bar\mu) + h_2(\la, \bar\mu).
\end{equation}
Restrict this relation to the royal variety:  if $\la= (2z, z^2)$ then
$$
\Phi_\omega(\la) =\frac{2\omega z^2 - 2z}{2-\omega 2z} = -z,
$$
and equation (\ref{relh}) yields for all $z,w \in \D$ and $\la=(2z,z^2), \
\mu=(2w, w^2),$
$$
\bar f(\mu) f(\la) = \frac{ h_1(\la, \bar\mu)}{1- \bar w z} +\frac{ h_2(\la, \bar\mu)}{1 - \bar w z}.
$$
The left hand side is a rank $1$ positive kernel on $\D$, hence lies on an extreme ray of $\mathcal{P}(\D)$.  The summands on the right hand side also belong to $\mathcal{P}(\D)$ and are thus non-negative constant multiples of $f^\vee f$.  Consequently there exists $c \in (0,1)$ such that
\begin{equation}
\label{hV}
h_1 = ch \qquad \mbox{ on     } \qquad\mathcal{V} \times \bar {\mathcal{V}}.
\end{equation}
Now consider the restriction of equation (\ref{relh}) to $F_\beta \times \bar F_\beta$ where $\beta \in \D$ and
$$
F_\beta = \{ (\beta z + \bar \beta, z) : z \in \D \}.
$$
Note that, by Theorem \ref{Ggeos}, the $F_\beta$ foliate $G$.  Furthermore
$$
\Phi_\omega(\beta z + \bar \beta, z) = \frac{2\omega z - \beta z - \bar\beta}{2 - \omega(\beta z + \bar \beta)} = \tau B_\al(z)
$$
where $\al = \bar\beta/(2\omega - \beta) \in \D$,
$$
\tau = \frac{2\omega -\beta}{2 - \omega\bar\beta} \in \T \mbox{ and }
B_\al(z) = \frac{z-\al}{1-\bar\al z}.
$$
 Hence, if $\la=(\beta z+\bar \beta,z), \ \mu=(\beta w+\bar\beta, w)$ with
$z,w \in \D$, equation (\ref{relh}) yields
$$
\bar f(\mu)\frac{1-\bar B_\al(\mu) B_\al(\la)}{1- \bar w z} f(\la) =
 \frac{ h_1(\la, \bar\mu)}{1- \bar w z} +\frac{ h_2(\la, \bar\mu)}{1 - \bar w z}.
$$
The left hand side is a rank $1$ positive kernel on $F_\beta$, and the summands on the right hand side are also positive kernels on $F_\beta$.  Hence there exists $t_\beta \in [0,1]$ such that
\begin{equation}
\label{hFb}
h_1 = t_\beta h \qquad \mbox{  on   } F_\beta \times \bar F_\beta.
\end{equation}
By Theorem \ref{Ggeos}, $F_\beta$ meets $\mathcal V$, and comparison of
equations (\ref{hV}) and (\ref{hFb}) now shows that $c=t_\beta$ for all $\beta \in \D$.
Hence $h_1 = ch$ on every $F_\beta \times \bar F_\beta$.  In particular,
$$
(h_1 -ch)(\la,\bar\la) = 0 \qquad \mbox{  for all   } \la \in G.
$$
It follows that all the Taylor coefficients of $h_1 - ch$ at zero vanish.  Thus $h_1 =ch$
on an open set, and so on all of $G\times\bar G$.  We have shown that if $h=h_1+h_2$ with $h_1, h_2 \in \Hered G$, then $h_1$ is a constant multiple of $h$.
That is, $h$ lies on an extreme ray of $\Hered G$.

We now turn to the proof of necessity in Theorem \ref{exrays}.  Let $h$ lie on an extreme ray of $\Hered  G$.  We have $h\neq 0$.  By Theorem \ref{realn} there exist a Hilbert space $H$, an $\mathcal{L}(H)$-valued spectral measure $E$ on $\mathbb{T}$  and a continuous function $u:\mathbb{T} \times G \to H$ such that 
 (i) and (ii) hold.    Let $J\subset \T$ be an interval.
 From property (ii) we have 
 \begin{equation}\label{splith}
 h =  \int_J + \int_{\T\setminus J}\left(1-\Phi^\vee_\omega \Phi_\omega\right) \langle E(d\omega) u_\omega, u_\omega\rangle,
 \end{equation}
 where $u_\omega(\lambda)$ denotes $u(\omega, \lambda)$.  This formula expresses $h$ as a sum of two elements of the cone $\Hered  G$, and since $h$ is supposed extremal, there exists $\nu(J) \geq 0$ such that
 \begin{equation}\label{defnu}
 \int_J \left(1 - \Phi_\omega^\vee \Phi_\omega\right) \langle E(d\omega )u_\omega, u_\omega\rangle  = \nu(J) h.
 \end{equation}
It is clear that $\nu$ is a countably additive set function on the intervals in $\T$, and so $\nu$ extends to a Borel probability measure on $\T$.

 For $\la \in G$ let 
\[
M(\la) =  \inf_{\omega\in\T} (1-|\Phi(\omega,\la)|^2).
\]
  Since $\Phi(.,\la)$ maps the closed unit disc into $\D$ we have $M(\la) > 0$.   From equation (\ref{defnu}) we have, for any interval $J$,
\begin{eqnarray*}
\nu(J)h(\la,\bar\la) &=&\int_J(1-|\Phi_\omega(\la)|^2) \langle E(d\omega) u(\omega,\la), u(\omega,\la) \rangle \\
   & \geq & M(\la) \int_J \langle E(d\omega) u(\omega,\la), u(\omega,\la) \rangle.
\end{eqnarray*}
On combining this inequality with Lemma \ref{causch} we find that, for any $\la,\mu \in G$ and any uniformly Lipschitz scalar function $\ph$ on $\T$,
\begin{eqnarray} \label{estE}
\left| \int_J \langle \ph(\omega) E(d\omega)u(\omega,\la), u(\omega,\mu) \rangle \right|^2 &\leq &
 \int_J |\ph(\omega)|^2 \langle E(d\omega)u(\omega,\la), u(\omega,\la) \rangle \cdot  \nn \\
  & & \qquad \int_J \langle E(d\omega)u(\omega,\mu), u(\omega,\mu) \rangle  \nn\\
   &\leq&  \sup_{\omega\in J}|\ph(\omega)|^2 \nu(J)^2 \frac {h(\la,\bar\la)h(\mu,\bar\mu)}{M(\la)M(\mu)}.
\end{eqnarray}

 Let $\omega_0$ be any point of the closed support of $\nu$ and let $\mathcal{J}$ denote the set of open intervals in $\T$ that contain $\omega_0$; thus $\nu(J) > 0$ for every $J\in\mathcal{J}$.  
 We claim that $h/(1-\Phi_{\omega_0}^\vee \Phi_{\omega_0})$ is positive semi-definite on $G$.  To see this fix $\lambda_1, \dots, \lambda_m\in G$ and $c_1, \dots, c_m \in \mathbb{C}$.   By equation (\ref{defnu}), for any $J\in\mathcal{J}$,
\begin{eqnarray*}
\lefteqn{\sum\limits_{i, j}\ \frac{h(\lambda_j, \bar \lambda_i)}{1 - \overline{\Phi_{\omega_0}(\lambda_i)} \Phi_{\omega_0}(\lambda_j)}\ \bar c_i c_j =} \\
  & &  \sum\limits_{i, j} \frac{\bar c_i c_j}{\nu(J)} \int_J \frac{1 - \overline{\Phi_{\omega}(\lambda_i)} \Phi_{\omega}(\lambda_j)}{1 - \overline{\Phi_{\omega_0}(\lambda_i)} \Phi_{\omega_0}(\lambda_j)}\ \langle E(d\omega) u(\omega,\la_j),u(\omega,\la_i )\rangle 
\end{eqnarray*}
and hence
\begin{eqnarray} \label{error}
\lefteqn{\left|\sum\limits_{i, j}\ \frac{h(\lambda_j, \bar \lambda_i)}{1 - \overline{\Phi_{\omega_0}(\lambda_i)} \Phi_{\omega_0}(\lambda_j)}\ \bar c_i c_j  - \frac{1}{\nu(J)} \int_J \langle E(d\omega) \sum_j c_j u(\omega, \la_j), \sum_i c_i u(\omega,\la_i) \rangle \right| =}   \nn\\
 & &  \left| \frac{1}{\nu(J)}\int_J \sum\limits_{i, j}  \left\{\frac{1 - \overline{\Phi_{\omega}(\lambda_i)} \Phi_{\omega}(\lambda_j)}{1 - \overline{\Phi_{\omega_0}(\lambda_i)} \Phi_{\omega_0}(\lambda_j)} -1 \right\}\bar c_i c_j  \langle E(d\omega) u(\omega,\la_j),u(\omega,\la_i) \rangle \right| \leq \nn \\
 & &  \frac{1}{\nu(J)} \sum\limits_{i, j} \left|\int_J  \left\{\frac{1 - \overline{\Phi_{\omega}(\lambda_i)} \Phi_{\omega}(\lambda_j)}{1 - \overline{\Phi_{\omega_0}(\lambda_i)} \Phi_{\omega_0}(\lambda_j)} -1 \right\}\bar c_i c_j  \langle E(d\omega) u(\omega,\la_j),u(\omega,\la_i) \rangle \right| \leq \nn \\
 & &  \sum\limits_{i, j} \sup_{\omega\in J} \left|\frac{1 - \overline{\Phi_{\omega}(\lambda_i)} \Phi_{\omega}(\lambda_j)}{1 - \overline{\Phi_{\omega_0}(\lambda_i)} \Phi_{\omega_0}(\lambda_j)} -1 \right| |c_ic_j| \left\{\frac{h(\la_j,\bar\la_j)h(\la_i,\bar\la_i)}{M(\la_j) M(\la_i)}\right\}^{\tfrac 12},
\end{eqnarray}
the last inequality by virtue of (\ref{estE}).

Let $\ep > 0$ and let 
\[
M= \max_{1\leq j\leq m} |c_j|^2\frac{h(\la_j, \bar\la_j)}{M(\la_j)}.
\]
 Choose $J\in\mathcal{J}$ so small that for all $\omega\in J$
\[
 \left|\frac{1 - \overline{\Phi_{\omega}(\lambda_i)} \Phi_{\omega}(\lambda_j)}{1 - \overline{\Phi_{\omega_0}(\lambda_i)} \Phi_{\omega_0}(\lambda_j)} -1 \right| < \frac{\ep}{m^2M}, \quad 1\leq i,j \leq m.
\]
By inequality (\ref{error}) we have
\[
\left| \sum\limits_{i, j}\ \frac{h(\lambda_j, \bar \lambda_i)}{1 - \overline{\Phi_{\omega_0}(\lambda_i)} \Phi_{\omega_0}(\lambda_j)}\ \bar c_i c_j  - A \right| < \ep
\]
where
\[
 A = \frac{1}{\nu(J)} \int_J \langle E(d\omega) \sum_j c_j u(\omega, \la_j), \sum_i c_i u(\omega,\la_i) \rangle  \geq 0.
\]
It follows that
\[
\sum\limits_{i, j}\ \frac{h(\lambda_j, \bar \lambda_i)}{1 - \overline{\Phi_{\omega_0}(\lambda_i)} \Phi_{\omega_0}(\lambda_j)}\ \bar c_i c_j  \geq 0.
\]
Thus $h\left/\left(1-\Phi^\vee_{\omega_0} \Phi_{\omega_0}\right)\right.$ is positive semi-definite on $G$ as claimed.  Since in addition $h$ lies on an extreme ray of $\Hered  G$, it follows that $h\left/\left(1-\Phi^\vee_{\omega_0} \Phi_{\omega_0}\right)\right.$ lies on an extreme ray of $\mathcal{P} (G)$.  Hence, by Proposition \ref{exposdef}, there is an analytic function $f: G\to \mathbb{C}$ such that  
$$
h=f^\vee\cdot \left(1-\Phi^\vee_{\omega_0} \Phi_{\omega_0}\right) \cdot f.
$$
  Thus necessity holds in Theorem \ref{exrays}. \qed

\begin{corollary}
The magic functions of $G$ are the functions  $m \circ \Phi_\omega$ where $m\in \Aut \D$ and $\omega\in\T$.  
\end{corollary}
Corollary \ref{Minvar} immediately yields the following invariance property.
\begin{corollary}\label{invar}
 The collection of functions 
 $$
 m\circ\Phi_\omega : G\to \mathbb{D},
$$
 where  $m\in\Aut\D$ and $\omega \in \mathbb{T}$,  is invariant under composition with automorphisms of $G$.
\end{corollary}

 \section{The automorphisms of $G$} \label{autG}
 In this section we prove that all automorphisms of $G$ are induced by M\"obius functions.
Note that any $m \in \Aut \mathbb{D}$ induces an automorphism $\tau(m)$ of $G$ by
$$
\tau(m) (z_1+z_2, z_1z_2) = (m(z_1) +m(z_2), m(z_1) m(z_2)),  \quad z_1, z_2 \in \mathbb{D}.
$$
\begin{theorem}\label{autos}
$\tau: \Aut \mathbb{D} \to \Aut G$ is an isomorphism of groups.
\end{theorem}
 
 It is immediate that $\tau$ is a homomorphism and is injective; we prove that $\tau$ is surjective by combining the invariance property of the $\Phi_\omega$ (Corollary \ref{invar}) with consideration of the action of automorphisms on certain geodesics.  Recall that the {\em Carath\'eodory distance} on a bounded domain $\Omega$ is the distance function on $\Omega$
 $$
C_\Omega (\lambda, \mu) = \sup\limits_F \rho (F(\lambda), F(\mu))
$$
 where $\rho$ is the pseudohyperbolic distance on $\mathbb{D}$ and the supremum is taken over all analytic functions $F: \Omega \to \mathbb{D}$.  A {\em complex geodesic} of $\Omega$ is an analytic function $\varphi: \mathbb{D} \to \Omega$ that is isometric with respect to $\rho$ and $C_\Omega$.  We identify the complex geodesics $\varphi$ and $\varphi \circ m$ for  any $m\in \Aut \mathbb{D}$.  It is clear that an automorphism of $\Omega$ induces a permutation of the complex geodesics of $\Omega$; in the case that $\Omega = G$, the complex geodesics are known explicitly, and we may deduce information about the automorphisms of $G$.
 
For $\beta, \lambda \in \mathbb{D}$ define
$$\varphi_\beta (\lambda) = (\beta \lambda + \bar \beta, \lambda) \in \mathbb{C}^2.$$
$\varphi_\beta$ maps $\mathbb{D}$ into $G$, and in fact $\varphi_\beta$ is a complex geodesic of $G$ \cite[Theorem 2.1]{AY3}, \cite[Theorem 0.1] {AY6}.  We call $\varphi_\beta$ a {\em flat geodesic} of $G$.  There are also complex geodesics of $G$ that are rational of degree 2 \cite[Theorem 2.2]{AY3}, \cite [Theorem 0.3]{AY6}, but here we only need the special one 
$\psi(\lambda) = (2\lambda, \lambda^2)$,
that is, the royal variety $\mathcal{V}$.  We shall need the following facts about these geodesics \cite[Theorem 5.5]{AY6}, \cite[Theorem 0.3] {AY7}.  

\begin{theorem}\label{Ggeos}
\begin{enumerate}
\item[\rm(1)]  Each point of $G$ lies on a unique flat geodesic.

\item[\rm(2)]  Each flat geodesic meets the royal variety exactly once.

\item[\rm(3)]  For $z_1, z_2 \in G$, the following are equivalent:
\begin{itemize}
\item[\rm(i)]  $\Phi_\omega$ is an extremal function for the Carath\'eodory extremal problem for $z_1, z_2$ for {\em every} $\omega \in\mathbb{T}$;

\item[\rm(ii)]  $z_1, z_2$ lie either on the royal variety or on a flat geodesic.
\end{itemize}
\end{enumerate}
\end{theorem}

Statement 3(i) means: for all $\omega \in\mathbb{T}$,
$$C_G(z_1, z_2) = \rho\left(\Phi_\omega (z_1), \Phi_\omega(z_2)\right).$$

\begin{lemma}\label{royal}
Every automorphism of $G$ maps the royal variety to itself and maps every flat geodesic to a flat geodesic.
\end{lemma}

\proof
Let $\alpha \in\Aut G$ and let $\psi$ be either the royal or a flat geodesic.  Consider any pair of points on the complex geodesic $\alpha \circ \psi$ of $G$ -- say $z_1=\alpha \circ \psi(\lambda_1), z_2 = \alpha \circ \psi(\lambda_2)$ where $\lambda_1 \neq \lambda_2$ in $\mathbb{D}$.  Observe that, by Theorem  \ref{Ggeos}, statement (3)(i),
$$
C_G (\psi(\lambda_1), \psi(\lambda_2)) = \rho(\Phi_\zeta \circ \psi(\lambda_1), \Phi_\zeta \circ \psi(\lambda_2))
$$
for all $\zeta \in \mathbb{T}$.  We claim that every $\Phi_\omega, \omega \in \mathbb{T}$, is a Carath\'eodory extremal function for $z_1, z_2$.  Indeed, for $\omega \in \mathbb{T}$, by virtue of Corollary \ref{invar}, there exist   $m\in\Aut\D$,  $\zeta \in \mathbb{T}$ such that 
$$
\Phi_\omega \circ \alpha =  m\circ \Phi_\zeta.
$$
Then
\begin{eqnarray*}
\rho(\Phi_\omega (z_1), \Phi_\omega(z_2)) &=& \rho(\Phi_\omega(\alpha \circ \psi(\lambda_1)), \Phi_\omega(\alpha \circ \psi(\lambda_2))\\
&=& \rho(m\circ \Phi_\zeta \circ \psi(\lambda_1), m\circ\Phi_\zeta \circ \psi(\lambda_2)) \\
&=& \rho(\Phi_\zeta \circ \psi(\lambda_1), \Phi_\zeta \circ \psi(\lambda_2))\\
&=& C_G (\psi(\lambda_1), \psi(\lambda_2))\\
&=& C_G(\alpha \circ \psi(\lambda_1), \alpha \circ\psi(\lambda_2))\\
&=& C_G(z_1, z_2),
\end{eqnarray*}
and $\Phi_\omega$ is a Carath\'eodory extremal as claimed.  Hence, by Theorem \ref{Ggeos}, statement (3)(ii), the geodesic $\alpha \circ \psi$ is either royal or flat.    Among the class of royal or flat geodesics, the royal variety is the unique one that meets more than one other geodesic in the class, and this property is preserved by automorphisms.  Hence if $\psi$ is the royal variety then so is $\alpha \circ \psi$, and $\alpha \circ \varphi_\beta$ is a flat geodesic for every $\beta \in \mathbb{D}$. 
\qed

\vspace{3mm}

\noindent{\bf Proof
of Theorem \ref{autos}}.
Let $\alpha$ be an automorphism of $G$.  By Lemma \ref{royal}, $\alpha$ maps the royal variety $ \mathcal{V}=\psi(\mathbb{D})$ to itself.  Consider  the case that $\alpha|\mathcal{V}$ is the identity:
$$\alpha (2\lambda, \lambda^2) = (2\lambda, \lambda^2), \qquad \lambda \in \mathbb{D}.$$
We shall show that $\alpha$ is the identity map on $G$.  By Corollary \ref{invar}, for each $\omega \in \mathbb{T}$ there exist  $m\in\Aut\D$   and $ \zeta_\omega \in \mathbb{T}$ such that   
\begin{equation}\label{Phioa}
\Phi_\omega \circ \alpha =  m\circ  \Phi_{\zeta_\omega}.
\end{equation}
Now, for any $\omega \in \mathbb{T}$ and $\lambda \in \mathbb{D}, \Phi_\omega (2\lambda, \lambda^2) =-\lambda$.  On applying equation (\ref{Phioa}) to a general point $(-2\lambda, \lambda^2)$ of $\mathcal{V}$ we obtain  $\lambda = m( \lambda).$ 
Thus  $\Phi_\omega\circ \alpha = \Phi_{\zeta_\omega}$.

Next consider the restriction of $\alpha$ to the flat geodesic $\varphi_0 (\mathbb{D})$ through $(0, 0)$.  Since $\alpha$ fixes $(0, 0)$ and maps $\varphi_0(\mathbb{D})$ to a flat geodesic $\varphi_\beta(\mathbb{D})$, it must be that $\beta =0$.  Hence $\alpha$ induces an automorphism of the disc $\{(0, \lambda): \lambda \in \mathbb{D}\}$ that fixes $(0, 0)$, and so there exists $\eta \in \mathbb{T}$ such that $\alpha(0, \lambda) = (0, \eta \lambda)$ for all $\lambda \in \mathbb{D}$.
We have
$$\zeta_\omega \lambda = \Phi_{\zeta_\omega} (0, \lambda) = \Phi_\omega \circ \alpha(0, \lambda) = \Phi_\omega(0, \eta \lambda) = \omega \eta \lambda,$$
and so $\Phi_\omega \circ \alpha = \Phi_{\eta \omega}$ for all $\omega \in \mathbb{T}$.  If $\alpha(s, p) = (s', p')$ we have 
$$
\frac {2\omega p' -s'}{2-\omega s'} = \frac{2\eta \omega p -s}{2-\eta \omega s}
$$
for all $\omega \in \mathbb{T}$.  On cross-multiplying and equating coefficients of powers of $\omega$ we find that $(s', p') = (s, p)$, that is, $\alpha$ is the identity map.

We have shown that if $\alpha|\mathcal{V}$ is the identity then $\alpha$ is the identity on $G$.  Now suppose that $\alpha$ induces the automorphism $m$ on $\mathcal{V}$, in the sense that 
$$\alpha(2\lambda, \lambda^2) = (2m(\lambda), m(\lambda)^2)$$
for all $\lambda \in \mathbb{D}$.  Then the automorphism $\tau(m^{-1})$ of $G$ satisfies
$$\tau(m^{-1}) (2m(\lambda), m(\lambda)^2) = (2\lambda, \lambda^2)$$ 
and hence $\tau(m^{-1}) \circ \alpha$ is an automorphism of $G$ that restricts to the identity on $\mathcal{V}$.  Thus $\tau(m^{-1})\circ \alpha$ is the identity, and so $\alpha= \tau(m)$.  We have shown that $\tau: \Aut \mathbb{D} \to \Aut G$ is a surjective map. \qed

\begin{corollary}\label{inhom}
$G$ is inhomogeneous and asymmetric.
\end{corollary}
\begin{proof}
Every element of $\Aut G$ preserves the royal variety and so $G$ is inhomogeneous.  The statement that $G$ is asymmetric means that some point of $G$ is not an isolated fixed point of an involutive isomorphism of $G$.  Suppose that $\tau(m), m\in\Aut\D$, is an involution that fixes $(0,0)$.  It is easy to see that either $\tau(m)$ is the identity or $\tau(m)(s,p)=(-s,p)$ for all $(s,p)\in G$.  In neither case is $(0,0)$ an isolated fixed point.  
\qed
\end{proof}

Jarnicki and Pflug \cite{JP} prove inhomogeneity by showing that $G$ is not isomorphic to $\D^2$ or the ball and appealing to the Cartan classification of bounded homogeneous domains in $\C^2$; they then deduce that (in our terminology) the orbit of $(0,0)$ is the royal variety $\mathcal{V}\cap G$ and thence show that $\tau$ is surjective.

We remark that $G$ only just fails to be symmetric:  any point $(z+w,zw) \in G\setminus \mathcal{V}$ is the unique fixed point of an involutive automorphism $\tau(m)$, where $m\in\Aut\D$ is chosen to satisfy $m(z)=w$.

\section{The Carath\'eodory distance}\label{cara}
If one can find an economical generating set for $\Hered \Omega$  consisting of hereditary functions of the form $1-f^\vee f$ then one can deduce a formula for the Carath\'eodory pseudodistance on $\Omega$.  The main idea here is in \cite{Ag2}.
\begin{theorem}
\label{caragen}
Let $\mathcal{M}$ be a set of holomorphic functions on $\Omega$ such that
\newline\mbox{$\{1-f^\vee f: f \in \mathcal{M}\}$} generates $\Hered \Omega$ and let $x,y \in \Omega$. 
Then 
\begin{equation}
\label{caraY}
C_\Omega(x,y) = \sup_{f \in \mathcal{M}} \rho(f(x),f(y)),
\end{equation}
and furthermore, if $\mathcal{M}$ is compact, then $\mathcal{M}$ contains a Carath\'eodory extremal
function for $x$ and $y$.
\end{theorem}

\proof
If $f\in \mathcal{M}$ then $1-f^\vee f \in \Hered \Omega$ and so $f \in \Hol(\Omega,\D)$. Thus
$$
\rho(f(x),f(y)) \le C_\Omega(x,y).
$$
Hence
$$
C_\Omega(x,y) \ge \sup_{f \in \mathcal{M}} \rho(f(x),f(y)).
$$

To prove the reverse inequality, write $x=(x_1,\dots,x_d), \ y=(y_1,\dots,y_d)$
and, corresponding to any normalised basis $u=(u_1,u_2)$ of a 2-dimensional
Hilbert space $H$, define $T(u)$ to be the commuting $d$-tuple $(T_1(u), \dots,T_d(u))$
of operators on $H$ where $T_j(u)$ is the operator whose matrix with respect to $u$
is $\mathrm{diag}(x_j,y_j)$.  A straightforward calculation \cite{Ag2}
shows that, for $f \in \Hol \Omega$,
\begin{equation}
\label{stcal}
||f(T(u))|| \le 1 \iff | <u_1,u_2> |^2 \le 1 - \rho(f(x),f(y))^2.
\end{equation}
It follows that
\begin{equation}
\label{caraform}
C_\Omega(x,y)^2 = 1 - \sup | <u_1, u_2> |^2
\end{equation} 
where the supremum on the right hand side is over all normalised bases $u$
of $H$ such that $\Omega$ is a spectral domain for $T(u)$.

Now pick a normalised basis $u$ of $H$ such that
\begin{equation} \label{chooseu}
| <u_1, u_2> |^2 = 1 - \sup_{f \in \mathcal{M}} \rho(f(x),f(y))^2.
\end{equation}
By the relation (\ref{stcal}), $||f(T(u))|| \le 1$ for all $f \in \mathcal{M}$,
and hence the closed conjugacy-invariant convex cone
$$
\{ h \in \Hol(\Omega \times \bar\Omega): h(T(u)) \ge 0 \}
$$
contains the set $E=\{1- f^\vee f: f \in \mathcal{M} \}.$  By hypothesis,
$E$ generates $\Hered \Omega$, and hence $h(T(u)) \ge 0$ for all
$h \in \Hered \Omega$.  In particular, $(1 - f^\vee f)(T(u)) \ge 0$
for all $f \in \Hol(\Omega, \D)$, or in other words, $\Omega$ is a 
spectral domain for $T(u)$.  By equations (\ref{caraform}) and (\ref{chooseu}),
\begin{eqnarray*}
C_\Omega(x,y)^2 &\le& 1 - |<u_1,u_2>|^2\\
	&=& \sup_{f \in \mathcal{M}} \rho(f(x),f(y))^2.
\end{eqnarray*}
Hence (\ref{caraY}) holds.  Clearly, if $\mathcal{M}$ is compact then the
supremum in formula $(\ref{caraY})$ is a maximum, which
is to say that some function in $\mathcal{M}$ is an extremal function
for the Carath\'eodory problem for $x$ and $y$.
\qed

We deduce a formula for the
Carath\'eodory distance on $G$ first given in \cite[Theorem 1.1, Corollary 3.5]{AY6}. 
\begin{corollary}
For any two points $x=(s_1,p_1),\ y=(s_2,p_2) \in G$,
\begin{eqnarray*}
C_G(x,y) &=& \sup_{\omega \in \T} \rho (\Phi_\omega(x), \Phi_\omega(y))\\
   &=& \sup_{\omega \in \T} \left| \frac{(s_2p_1-s_1p_2)\omega^2+2(p_2-p_1)\omega+s_1-s_2}{(s_1-\bar s_2p_1)\omega^2-2(1-p_1\bar p_2)\omega + \bar s_2-s_1\bar p_2} \right |.
\end{eqnarray*}
\end{corollary}
Thus, for every pair of points in $G$ there is a Carath\'eodory
extremal of the form $\Phi_\omega$ for some $\omega \in \T$.  The proof is immediate from Theorem \ref{caragen} and Corollary \ref{compactgen}.
\begin{proposition}\label{minimal}
The set 
$$
\mathcal{Y}=\{1- \Phi^\vee_\omega \Phi_\omega : \omega \in \T \}
$$
is a minimal closed generating set for $\Hered G$.
\end{proposition}
\begin{proof}
By Corollary \ref{compactgen}, $\mathcal{Y}$ is a compact generating set.  Suppose some proper closed subset $C$ of $\mathcal{Y}$ generates $\Hered G$.  Pick $\eta\in\T$ such that $1-\Phi^\vee_\eta\Phi_\eta \notin C$.  By \cite[Theorem 1.6]{AY7} there exist points $x,y\in G$ such that $\Phi_\omega$ is a Carath\'eodory extremal for $x,y$ uniquely for $\omega=\eta$: for example, we could take $x=(0,0),\, y=(\bar\eta,0)$.  We then have
\[
 \sup_{1- \Phi^\vee_\omega \Phi_\omega \in C} \rho(\Phi_\omega(x),\Phi_\omega(y)) < \rho(\Phi_\eta(x),\Phi_\eta(y)) = C_G(x,y).
\]
In view of Theorem \ref{caragen} this contradicts the assumption that $C$ generates $\Hered G$. \qed
\end{proof}

For all the domains $\D$, $\D^2$ and $G$   there is a compact set $\mathcal{M}$ of magic functions with  the property that $\{1-f^\vee f: f\in\mathcal{M}\}$ generates the hereditary cone.  Hence, for each of these domains, for any pair of distinct points in the domain there is a magic function that is a Carath\'eodory extremal.

\end{document}